\documentclass[12pt]{article}
\usepackage{amsmath,amsthm,amssymb,fullpage,enumerate,cleveref,url}
\newtheorem{theorem}{Theorem}[section]
\newtheorem{claim}[theorem]{Claim}
\newtheorem{lemma}[theorem]{Lemma}
\newtheorem{definition}[theorem]{Definition}
\newtheorem{corollary}[theorem]{Corollary}

\newtheorem{conjecture}[theorem]{Conjecture}

\usepackage{todonotes}
\newcommand{\eps}{\varepsilon}

\newcommand{\F}{\mathcal{F}}
\newcommand{\U}{\mathcal{U}}
\newcommand{\s}{\mathbf{s}}
\newcommand{\Q}{\mathbb{Q}}
\newcommand{\B}{\mathcal{B}}
\newcommand{\E}{\mathbb{E}}
\DeclareMathAccent{\widehat}{\mathord}{largesymbols}{"62}

\title{Improved bounds for the sunflower lemma}
\author{
Ryan Alweiss\thanks{Research supported by an NSF Graduate Research Fellowship.}\\
Department of Mathematics\\
Princeton University\\
\texttt{alweiss@math.princeton.edu}
\and
Shachar Lovett\thanks{Research supported by NSF award 1614023.}\\
Department of Computer Science\\
University of California, San Diego\\
\texttt{shachar.lovett@gmail.com}
\and 
Kewen Wu\\
Department of Computer Science\\
University of California, Berkeley\\
\texttt{shlw\_kevin@hotmail.com}
\and
Jiapeng Zhang\\
Department of Computer Science\\
University of Southern California\\
\texttt{jpeng.zhang@gmail.com}
}

\begin{document}
\maketitle

\begin{abstract}
A sunflower with $r$ petals is a collection of $r$ sets so that the intersection of each pair is equal to the intersection of all of them. Erd\H{o}s and Rado proved the sunflower lemma: for any fixed $r$, any family of sets of size $w$, with at least about $w^w$ sets, must contain a sunflower with $r$ petals. The famous sunflower conjecture states that the bound on the number of sets can be improved to $c^w$ for some constant $c$. In this paper, we improve the bound to about $(\log w)^w$. In fact, we prove the result for a robust notion of sunflowers, for which the bound we obtain is sharp up to lower order terms.
\end{abstract}

\section{Introduction}

\label{sec:introduction}

Let $X$ be a finite set. A set system $\F$ on $X$ is a collection of subsets of $X$. We call $\F$ a $w$-set system if each set in $\F$ has size at most $w$.

\begin{definition}[Sunflower]
A collection of sets $S_1,\ldots,S_r$ is an \emph{$r$-sunflower} if
$$
S_i \cap S_j = S_1 \cap \cdots \cap S_r, \qquad \forall i \ne j.
$$
We call $K=S_1 \cap \cdots \cap S_r$ the \emph{kernel} and $S_1 \setminus K,\ldots,S_r \setminus K$ the \emph{petals} of the sunflower.
\end{definition}

Erd\H{o}s and Rado \cite{ErdosR1960} proved that large enough set systems must contain a sunflower.  It is noteworthy that Erd\H{o}s and Rado originally called sunflowers $\Delta$-systems, but the term ``sunflower'' was coined by Deza and Frankl \cite{deza1981every} and is now more widely used.

\begin{lemma}[Sunflower lemma \cite{ErdosR1960}]
\label{lemma:sunflower}
Let $r \ge 3$ and $\F$ be a $w$-set system of size $|\F| \ge w! \cdot (r-1)^w$. Then $\F$ contains an $r$-sunflower.
\end{lemma}

Erd\H{o}s and Rado conjectured in the same paper that the bound in \Cref{lemma:sunflower} can be drastically improved.

\begin{conjecture}[Sunflower conjecture \cite{ErdosR1960}]
\label{conj:sunflower}
Let $r \ge 3$. There exists $c=c(r)$ such that any $w$-set system $\F$ of size $|\F| \ge c^w$ contains an $r$-sunflower.
\end{conjecture}
The bound in \Cref{lemma:sunflower} is of the form $w^{w(1+o(1))}$ where the $o(1)$ term depends on $r$. Despite nearly 60 years of research, the best known bounds were still of the form $w^{w(1+o(1))}$, even for $r=3$.  Kostochka \cite{kostochka1997bound} proved that any $w$-set system of size $|\F| \ge c w! \cdot (\log \log \log w / \log \log w)^w$ must contain a $3$-sunflower for some absolute constant $c$. Recently, Fukuyama \cite{fukuyama2018improved} claimed an improved bound of $w^{(3/4+o(1))w}$ for $r=3$, but this proof has not been verified.

In this paper, we vastly improve the known bounds. We prove that any $w$-set system of size
$(\log w)^{w(1+o(1))}$ must contain a sunflower. More precisely, we obtain the following:

\begin{theorem}[Main theorem, sunflowers]
\label{thm:sunflower}
Let $r \ge 3$. For some constant $C$, any $w$-set system $\F$ of size $|\F| \ge (Cr^3 \log w \log \log w)^w$ contains an $r$-sunflower.
\end{theorem}

We will implicitly assume throughout the paper that $\log\log w > 0$. Formally, to handle the case of $w=2$, we interpret $\log$ as logarithm in base $1.9$. 

\subsection{Robust sunflowers}

We consider a ``robust'' generalization of sunflowers, the study of which was initiated by Rossman \cite{Rossman14}, who was motivated by questions in complexity theory. Later, it was studied by Li, Lovett and Zhang \cite{li2018sunflowers} in the context of the sunflower conjecture.

First, we define a more ``robust'' version of the property of having disjoint sets. Given a finite set $X$, we denote by $\U(X,p)$ the distribution over subsets $R \subset X$, where each element $x \in X$ is included in $R$ independently with probability $p$
(this is sometimes referred to as a ``$p$-biased distribution'').

\begin{definition}[Satisfying set system]
Let $0 < \alpha,\beta < 1$. A set system $\F$ on $X$ is $(\alpha,\beta)$-satisfying if
$$
\Pr_{R \sim \U(X,\alpha)}[\exists S\in \F, S\subset R] > 1-\beta.
$$
\end{definition}

The explanation for the name ``satisfying'' is that if the set system is interpreted as a disjunctive normal form (DNF) formula, then this condition is that the formula has more than a $1-\beta$ probability of being satisfied  on $\alpha$-biased inputs; see \cite{lovett2019dnf}. As mentioned before, the property of being satisfying is a robust analogue of the property of having disjoint sets.

\begin{lemma}[{\cite[Claim 14]{lovett2019dnf}}]
\label{lemma:satisfying_contains_disjoint}
If $\F$ is a $(1/r,1/r)$-satisfying set system and $\emptyset\notin\F$, then $\F$ contains $r$ pairwise disjoint sets.
\end{lemma}

\begin{proof}
Let $\F$ be a set system on $X$.
Consider a random coloring of elements of $X$ with $r$ colors, so that each $x \in X$ has an equal probability of being colored with any of the $r$ colors, independently of the other elements.
For $1 \le i \le r$, let $Y_i$ denote the subset of $X$ colored with color $i$ and let $\mathcal E_i$ denote the event that $\F$ contains an $i$-monochromatic set, namely,
$$
\mathcal E_i = \left[ \exists S \in \F, S \subset Y_i \right].
$$
Note that $Y_i \sim \U(X,1/r)$, and since we assume $\F$ is $(1/r,1/r)$-satisfying, we have
$$
\Pr[\mathcal E_i] > 1-1/r.
$$
By the union bound, with positive probability all of $\mathcal E_1,\ldots,\mathcal E_r$ hold. In this case, $\F$ contains a set which is $i$-monochromatic
for each $i=1,\ldots,r$. Such sets must be pairwise disjoint.
\end{proof}

Given a set system $\F$ on $X$ and a set $T \subset X$, the \emph{link} of $\F$ at $T$ is
$$
\F_T = \{S \setminus T: S \in \F, T \subset S\}.
$$

We now formally define a \emph{robust sunflower} (which was called a \emph{quasi-sunflower} in \cite{Rossman14} and an \emph{approximate sunflower} in \cite{lovett2019dnf}).

\begin{definition}[Robust sunflower]
Let $0 < \alpha,\beta < 1$, $\F$ be a set system, and let $K = \bigcap_{S \in \F} S$ be the common intersection of all sets in $\F$. $\F$ is an $(\alpha,\beta)$-robust sunflower if (i) $K \notin \F$, and (ii) $\F_K$ is $(\alpha,\beta)$-satisfying. We call $K$ the \emph{kernel}.
\end{definition}

The kernel is not allowed to appear in the set system for technical reasons; we do not want a set system consisting of a single set to be a robust sunflower.  In order to circumvent this technical obstacle, when relevant we will ensure that the set systems we work with are $w$-uniform for some $w$; i.e., all sets in the system will be of size $w$.  Any $w$-set system can be turned into a $w$-uniform system by adding $w-|S|$ distinct dummy elements to the sets $S$ of the system with $|S|<w$ so that no dummy element appears in more than one set.  

Robust sunflowers are a generalization of satisfying set systems. In particular, an $(\alpha,\beta)$-satisfying $w$-uniform system $\F$ will be an $(\alpha,\beta)$-robust sunflower as long as $\F$ contains at least two distinct sets, because if $R \supset S$ for some $S \in F$, then clearly $R \supset S \setminus K$.

\begin{lemma}[{\cite[Corollary 15]{lovett2019dnf}}]
\label{lemma:robust_sunflower_contains_sunflower}
Any $(1/r,1/r)$-robust sunflower contains an $r$-sunflower.
\end{lemma}

\begin{proof}
Let $\F$ be a $(1/r,1/r)$-robust sunflower, and let $K=\bigcap_{S \in \F} S$ be the common intersection of the sets in $\F$.
Note that by assumption, $\F_K$ does not contain the empty set as an element.
\Cref{lemma:satisfying_contains_disjoint} gives that $\F_K$ contains $r$ pairwise disjoint sets $S_1,\ldots,S_r$. Thus $S_1 \cup K,\ldots,S_r \cup K$ is an $r$-sunflower in $\F$.
\end{proof}

The proof of \Cref{thm:sunflower} follows from the following stronger theorem, by setting $\alpha=\beta=1/r$ and applying \Cref{lemma:robust_sunflower_contains_sunflower}. The theorem verifies a conjecture raised in \cite{lovett2019dnf}, and answers a question of \cite{Rossman14}.

\begin{theorem}[Main theorem, robust sunflowers]
\label{thm:robust_sunflower}
Let $0 < \alpha,\beta < 1$. For some constant $C$, any $w$-uniform set system $\F$ of size $|\F| \ge \left(\frac{C}{\alpha^2} \cdot \left(\log w \log \log w+ \left(\log\frac 1\beta\right)^2\right)\right)^w$ contains an $(\alpha,\beta)$-robust sunflower.
\end{theorem}

For fixed $\alpha,\beta$, 
the bound of $(\log w)^{w(1+o(1))}$ for robust sunflowers in \Cref{thm:robust_sunflower} is sharp; it cannot be improved beyond $(\log w)^{w(1-o(1))}$. We give an example demonstrating this in \Cref{lemma:lower_bound}. 

\subsection{Proof overview}

We now outline our proof of \Cref{thm:robust_sunflower}. First, we define \emph{spread} set systems.

\begin{definition}[Spreadness, \cite{lovett2019dnf}] 
\label{def: protospread} We say that a $w$-set system $\F$ is \emph{$\kappa$-spread} if $|\F| \ge \kappa^w$ and $|\F_T| \le \kappa^{-|T|}|\F|$ for all non-empty $T$, where $|\F_T|$ is the size of the link at $T$. 
\end{definition}

The paper \cite{lovett2019dnf} calls these ``regular set systems'', but we use the more descriptive term ``spread''.  Furthermore, we will ultimately need to slightly generalize the above definition so that it depends on more than one parameter; see \Cref{def: spread_set_system} for more details.

Say $\F$ is a $w$-set system of size $|\F| \ge \kappa^w$ on a ground set $X$. Then either $\F$ is $\kappa$-spread, or there is a link $\F_T$ of size $|\F_T| \ge \kappa^{w-|T|}$. In the latter ``structured'' case, we can simply pass to the link and apply induction, much like in the original proof of Erd\H{o}s and Rado \cite{ErdosR1960}.   

Thus, it suffices to consider the ``pseudorandom'' case of $w$-set systems which are $\kappa$-spread. In \cite{li2018sunflowers, lovett2019dnf}, it was conjectured that for some absolute $C$, a $\kappa$-spread $\F$ is necessarily $(1/3,1/3)$-satisfying if $\kappa \ge (\log w)^C$. We show that there is some $\kappa=(\log w)^{1+o(1)}$ which is sufficient (see \Cref{thm:kappa_bound}), completing the proof of \Cref{thm:robust_sunflower}.  We also show that this value of $\kappa$ is tight, up to the $o(1)$ in the exponent.

We next outline how we obtain the bound on $\kappa$. We will prove that a $\kappa$-spread $w$-set system is $(\alpha,\beta)$-satisfying for appropriate $\kappa, w, \alpha, \beta$ through a series of reductions.  Let $\F$ be a $w$-set system which is $\kappa$-spread. Sample $W \sim \U(X,p)$ for some $p=O(1/\log w)$. We show that with high probability over the choice of $W$, for almost all sets $S \in \F$, there exists a set $S' \in \F$ such that: (i) $S' \setminus W \subset S \setminus W$; and (ii) $|S' \setminus W| \le w'$, for some $w'$ which we will take to be $w(1-\eps)$ for a small $\eps$. We throw out sets $S$ that do not satisfy this property (which we call the \textit{bad} sets for $W$), and replace any $S$ that does with $S' \setminus W$ (we call these the \textit{good} sets for $W$). This yields a $w'$-set system $\F'$ which has almost as many sets as the original $\F$, and therefore will have almost the same spreadness (since the best $\kappa$ for which $|\F| \ge \kappa^T|\F_T|$ and for which $|\F'| \ge \kappa^T|\F_T|$ will be similar).  \footnote{In the language of DNFs, we take a random restriction and approximate the result by a smaller width DNF whose clauses come from removing some variables from clauses of the original DNF.} If we sample $W' \sim \U(X,p')$ and find some $S' \setminus W \in \F'$ such that $S' \setminus W \subset W'$, then $S' \subset W \cup W'$.  Since $W \cup W' \sim \U(x,q)$ for $q=p+p'-pp'$, the satisfyingness of $\F$ will follow from the satisfyingness of $\F'$, although with different values of $w, \kappa, \alpha, \beta$.  We iterate, creating $W \cup W' \cup W'' \cup \cdots$, which will be a random set.  In particular, applying this ``reduction step'' (\Cref{lemma:reduction}) iteratively $t=\log w$ times with the correct parameters yields a set system with much smaller sets, and then we can apply Janson's inequality (see \Cref{lemma:final}) to finish.  In order to obtain our nearly optimal bounds, we must use a slight generalization of spreadness; see \Cref{def: spread_set_system}.  The parameters change throughout the algorithm; we will specify exactly how later.

To conclude, let us comment on how we prove the key reduction step (\Cref{lemma:reduction}). The main idea is to use an encoding argument, inspired by Razborov's proof of H\aa stad's switching lemma \cite{Hastad87,razborov1995bounded}. We show that pairs $(W,S)$ for which $S$ is bad for $W$ can be efficiently encoded, crucially relying on the spreadness condition. This allows to show that for a random $W$ it is very unlikely that there will be many bad sets. This ensures that the $w'$-set system $\F' = \{S' \setminus W: S' \in \F, |S' \setminus W| \le w'\}$ will have almost as many sets as $\F$.

\paragraph{Paper organization.}
We present the proof of \Cref{thm:robust_sunflower} in \Cref{sec:proof_main}. We prove a matching lower bound for robust sunflowers in \Cref{sec:lower_bound}. We discuss some applications in \Cref{sec:applications}. We introduce several conjectures aimed at proving the sunflower conjecture in \Cref{sec:rainbow}.

\section{Proof of \Cref{thm:robust_sunflower}}
\label{sec:proof_main}

We proceed to prove \Cref{thm:robust_sunflower}. Let $\F$ be a set system, and let $\sigma:\F \mapsto \Q_{\ge 0}$ be a weight function that assigns nonnegative rational weights to sets in $\F$ which are not all $0$. A weight function is necessary because in the reduction multiple sets $S$ can be replaced by the same $S'$, and we must weight such $S'$ accordingly.  For simplicity, we do not permit irrational weights. We call the pair $(\F,\sigma)$ a \emph{weighted set system}. For a subset $\F' \subset \F$ we write $\sigma(\F') = \sum_{S \in \F'} \sigma(S)$
for the sum of the weights of the sets in $\F'$. 

A \emph{weight profile} is a vector $\s=(s_0; s_1,\ldots,s_k)$ where $s_0 \ge s_1 \ge \cdots \ge s_k \ge 0$ and $s_0>0$. The $s_i$ need not be rational. We assume implicitly that $s_i=0$ for all $i>k$.

\begin{definition}[Spread weighted set system]
\label{def: spread_set_system}
Let $\s=(s_0;s_1,\ldots,s_w)$ be a weight profile. A weighted set system $(\F,\sigma)$ is $\s$-spread if
\begin{enumerate}[(i)]
\item $\sigma(\F) \ge s_0$;
\item $\sigma(\F_T) \le s_{|T|}$ for any link $\F_T$ with non-empty $T$.
\end{enumerate}
In particular, $\F$ is a $w$-set system.
\end{definition}

\begin{definition}[Spread set system]
Let $\s$ be a weight profile. A set system $\F$ is $\s$-spread if there exists a weight function $\sigma:\F \mapsto \Q_{\ge 0}$ such that $(\F,\sigma)$ is $\s$-spread.
\end{definition}

This generalizes \Cref{def: protospread}; we call a set system $\kappa$-spread if it is $(1;\kappa^{-1},\ldots,\kappa^{-w})$-spread for the uniform weight function.

We note that one may always normalize a weight profile to have $s_0=1$. However, keeping $s_0$ as a free parameter helps to simplify some of the arguments later.

The main idea is to show that set systems which are $\s$-spread, for $\s$ appropriately chosen, are ``random looking'' and in particular must be $(\alpha,\beta)$-satisfying. This motivates the following definition.

\begin{definition}[Satisfying weight profile]
Let $0<\alpha,\beta<1$. A weight profile $\s$ is $(\alpha,\beta)$-satisfying if any $\s$-spread set system is $(\alpha,\beta)$-satisfying.
\end{definition}

The following lemma concerning satisfying weight profiles, which appears in \cite{lovett2019dnf}, underlies our proof of \Cref{thm:robust_sunflower}. We repeat the proof here for completeness.

\begin{lemma}
\label{lemma:weight_profile_robust_sunflower}
Let $0<\alpha,\beta<1$ and $w \ge 2$. Let $\kappa=\kappa(w) > 1$ be a (not necessarily strictly) increasing function of $w$ such that the weight profile $(1;\kappa^{-1},\ldots,\kappa^{-w})$ is $(\alpha,\beta)$-satisfying. Then any $w$-uniform set system $\F$ of size $|\F| > \kappa^w$ must contain an $(\alpha,\beta)$-robust sunflower. 
\end{lemma}

\begin{proof}
Assume the contrary, and let $\F$ be a $w$-uniform set system on $X$ of size $|\F|>\kappa^w$ without an $(\alpha,\beta)$-robust sunflower. If the uniform distribution $\sigma(S)=1/|\F|$ for $S \in \F$ is $(1;\kappa^{-1},\ldots,\kappa^{-w})$-spread, then $\F$ is $(\alpha,\beta)$-satisfying and thus is an $(\alpha,\beta)$-robust sunflower.

Thus, there exists some nonempty $K \subset X$ so that $|\F_K| > \kappa^{-|K|}|\F|$.  Pick a maximal such $K$. Note that we cannot have $|K|=w$, as otherwise $|\F_K|=1 < \kappa^{-w}|\F|$, and so we must have that $1 \le |K| \le w-1$. Let $\F' = \F_K$. Note that $|\F'|>\kappa^{w-|K|}$; and for any non-empty set $T$ disjoint from $K$, $|\F'_T| = |\F_{K \cup T}| \le \kappa^{-|K|-|T|}|\F| \le \kappa^{-|T|}|\F'|$, by the maximality of $|K|$. Let $\sigma'(S)=1/|\F'|$ for $S \in \F'$. Then $(\F',\sigma')$ is $(1;\kappa^{-1},\ldots,\kappa^{-\ell})$-spread for $\ell=w-|K|$, where $\kappa=\kappa(w) \ge \kappa(\ell)$.  Thus $\F'$ is $(\alpha,\beta)$-satisfying, and hence $\{S \cup K: S \in \F'\}$ is an $(\alpha,\beta)$-robust sunflower contained in $\F$.
\end{proof}

\Cref{lemma:weight_profile_robust_sunflower} motivates the following definition. For $0<\alpha,\beta<1$ and $w \ge 2$, let $\kappa(w,\alpha,\beta)$ be the least $\kappa$ such that $(1;\kappa^{-1},\ldots,\kappa^{-w})$ is $(\alpha,\beta)$-satisfying. \Cref{thm:robust_sunflower} follows by combining \Cref{lemma:weight_profile_robust_sunflower} with the following theorem, which bounds $\kappa(w,\alpha,\beta)$.

\begin{theorem}
\label{thm:kappa_bound}
$\kappa(w,\alpha,\beta) = O\left(\frac{1}{\alpha^2} \cdot \left(\log w \log \log w+ \left(\log\frac 1\beta\right)^2\right)\right)$.
\end{theorem}

To prove \Cref{thm:kappa_bound}
we must show that, if $\kappa$ is the right-hand side of the above expression the following holds. Let $\F$ be
a weighted $w$-set system with total weight $1$ and where the weight of any link of a nonempty set $T$ is at most $\kappa^{-|T|}$.  Then $\F$ is $(\alpha,\beta)$-satisfying, i.e. the chance that some set from $\F$ is contained in a subset of the ground set $X$ sampled an $\alpha$-biased distribution is at least $1-\beta$.   We prove \Cref{thm:kappa_bound} in the remainder of this section. 

\subsection{A reduction step}

Let $\F$ be a $w$-set system on $X$, and fix $w' \le w$. The main goal in this section is to reduce our problem for $\F$ to the same problem for some $w'$-set system $\F'$. We prove the following lemma in this section.

\begin{lemma}
\label{lemma:reduction}
Let $\s=(s_0;s_1,\ldots,s_w)$ be a weight profile, $w' \le w$, $\delta>0$ and define $\s'=((1-\delta) s_0; s_1,\ldots,s_{w'})$. Assume $\s'$ is $(\alpha',\beta')$-satisfying. Then for any $p>0$, $\s$ is $(\alpha,\beta)$-satisfying for
$$
\alpha = p + (1-p)\alpha', \quad \beta = \beta' + \frac{(4/p)^w s_{w'}}{\delta s_0}.
$$
\end{lemma}

In order to prove this lemma, we need the notions of good and bad sets with respect to some $W \subset X$. 

\begin{definition}

Let $W \subset X$.  Given a set $S \in \F$, the pair $(W,S)$ is said to be \emph{good} if there exists a set $S' \in \F$ (possibly with $S'=S$) such that
\begin{enumerate}[(i)]
\item $S' \setminus W \subset S \setminus W$;
\item $|S' \setminus W| \le w'$.
\end{enumerate}

If no such $S'$ exists, we say that $(W,S)$ is \emph{bad}. 

\end{definition}

Note that if $W$ contains a set in $\F$ (i.e., $S' \subset W$ for some $S' \in \F$) then all pairs $(W,S)$ are good.

\begin{lemma}
\label{lemma:reduction_fixed_W}
Let $(\F,\sigma)$ be an $\s=(s_0;s_1,\ldots,s_w)$-spread weighted set system on $X$. Let $W \subset X$ be a uniformly random subset of size $|W|=p|X|$ and $\B(W) = \{S \in \F: (W,S) \text{ is bad}\}$. Then $\E_W [\sigma(\B(W))] \le (4/p)^w s_{w'}$.
\end{lemma}

\begin{proof}
First, we simplify the setting a bit. We may assume by scaling $\sigma$ and $\s$ by the same factor that $\sigma(S)=N_S$ is an integer for each $S \in \F$. Let $N=\sum N_S \ge s_0$. By scaling back down by $N$, we can then identify $(\F,\sigma)$ with the uniform distribution on the multi-set system $\F'=\{S_1,\ldots,S_N\}$, where each set $S \in \F$ is repeated $N_S$ times in $\F'$. Thus
$$
\sigma(\B(W)) = |\{i : S_i\in\F'\text{ and }(W,S_i) \text{ is bad}\}|.
$$

Assume that $(W,S_i)$ is bad in $\F'$. In particular, this means that $W$ does not contain any set in $\F$.
We describe $(W,S_i)$ with a small amount of information. Let $|X|=n$ and $|W|=pn$. We encode $(W,S_i)$ as follows:
\begin{enumerate}
\item The first piece of information is $W \cup S_i$, a subset of $X$ of size between $pn$ and $pn+w$. The number of options for this is at most:
$$
\sum_{i=0}^{w}{\binom n {pn+i}} \le \sum_{i=0}^{w}\left(\frac{1-p}{p}\right)^i{\binom n {pn}} \le \left(\frac{1-p}{p}+1\right)^{w} {\binom n {pn}}=p^{-w}{\dbinom n {pn}}.
$$
\item Given $W \cup S_i$, let $j$ be minimal such that $S_j \subset W \cup S_i$, so that $j$ depends only on $W \cup S_i$. Given $S_j$, there are fewer than $2^w$ possibilities for $A=S_i \cap S_j$.  As such, we will let $A$ be the second piece of information. 
\item Because $S_j \subset W \cup S_i$ and $A = S_i\cap S_j$, we have $S_j \setminus W \subset A$. Since $(W,S_i)$ is bad, $|A|\ge |S_j \setminus W|>w'$.
The number of the sets in $\F'$ which contain $A$ is $|\F'_A| \le s_{w'}$.  The third piece of information will be which one of these is $S_i$.
\item Finally, once we have specified $S_i$, we will specify $S_i \cap W$, which is one of $2^w$ possible subsets of $S_i$.  
\end{enumerate}

From these four pieces of information one can uniquely reconstruct $(W,S_i)$. Thus the total number of bad pairs $(W,S_i)$ is bounded by
$$
p^{-w} {\binom n {pn}} \cdot 2^w \cdot s_{w'} \cdot 2^w = (4/p)^w s_{w'} {\binom n {pn}}.
$$
Because the number of sets $W \subset X$ of size $|W|=p|X|$ is $\binom n {pn},$ the lemma follows by taking expectation over $W$.
\end{proof}

The following is a corollary of \Cref{lemma:reduction_fixed_W}, where we replace sampling $W \subset X$ of size $|W|=p|X|$ with sampling $W \sim \U(X,p)$.

\begin{corollary}
\label{cor:reduction_prob_W}
Let $(\F,\sigma)$ be an $\s=(s_0;s_1,\ldots,s_w)$-spread weighted set system on $X$. Let $W \sim \U(X,p)$ and $\B(W) = \{S \in \F: (W,S) \text{ is bad}\}$. Then $\E_W [\sigma(\B(W))] \le (4/p)^w s_{w'}$.
\end{corollary}

\begin{proof}
The proof is by a reduction to \Cref{lemma:reduction_fixed_W}.
Replace the base set $X$ with a much larger set $X'$ (without changing $\F$, so the new elements do not belong to any set in $\F$). Let $W' \subset X'$ be a uniform set of size $|W'|=p|X'|$, and let $W=W' \cap X$. Then as the size of $X'$ goes to infinity, the distribution of $W$ approaches $\U(X,p)$.\end{proof}

\begin{proof}[Proof of \Cref{lemma:reduction}]
Let $(\F,\sigma)$ be an $\s=(s_0;s_1,\ldots,s_w)$-spread weighted set system on $X$. Let $W \sim \U(X,p)$. Say that $W$ is $\delta$-bad if $\sigma(\B(W)) \ge \delta s_0$. By applying \Cref{cor:reduction_prob_W} and Markov's inequality, we obtain that
$$
\Pr[W \text{ is } \delta\text{-bad}] \le \frac{\E[\sigma(\B(W))]}{\delta s_0} \le \frac{(4/p)^w s_{w'}}{\delta s_0}.
$$
Fix $W$ which is not $\delta$-bad. If for some $S \in \F$, the pair $(W,S)$ is good, then there exists $\pi(S)=S' \in \F$ (possibly with $S'=S$) such that (i) $S' \setminus W \subset S \setminus W$ and (ii) $|S' \setminus W| \le w'$.

Define a new weighted set system $(\F',\sigma')$ on $X'=X \setminus W$ as follows:
$$
\F' = \{\pi(S) \setminus W: S \in \F \setminus \B(W)\}, \quad \sigma'(S' \setminus W) = \sigma(\pi^{-1}(S')).
$$
We claim that $\F'$ is $\s'=((1-\delta)s_0;s_1,\ldots,s_{w'})$-spread. To see that, note that $\sigma'(\F')=\sigma(\F \setminus \B(W)) \ge (1-\delta)s_0$ and that for any set $T \subset X'$, 
$$
\sigma'(\F'_T) = \sum_{S' \supset T} \sigma'(S') = \sum_{S: \pi(S) \supset T} \sigma(S) \le \sum_{S \supset T} \sigma(S) = \sigma(\F_T) \le s_{|T|}.
$$ 
Finally, all sets in $\F'$ have size at most $w'$. Hence, by the assumption $\s'$ is $(\alpha',\beta')$-satisfying. Thus, if we choose $W' \sim \U(X',\alpha')$ then we obtain that with probability more than $1-\beta'$, there exists $S^* \in \F'$ such that $S^* \subset W'$. Recall that $S^*=S \setminus W$ for some $S \in \F$. Thus $S \subset W \cup W' \sim \U(X,p+(1-p)\alpha')$. The value of $\beta$ accounts for the case where $W$ is $\delta$-bad.\end{proof}

\subsection{A final step}

Repeated applications of the reduction step yield a set system which is almost as spread as the original set system, but whose sets are much smaller.  Thus the spreadness guarantee becomes good compared to the size of the sets in the system. In this case, we can just directly show that the set system is satisfying. A similar argument appears in \cite{Rossman14}.

\begin{lemma}
\label{lemma:final}
Let $0<\alpha,\beta<1$, $w \ge 2$, and set $\kappa=\max\left\{4\log(1/\beta),2\right\} \cdot w / \alpha$.
Let $(\F,\sigma)$ be an $\s=(s_0;s_1,\ldots,s_w)$-spread weighted set system where $s_i < \kappa^{-i} s_0$. 
Then $\F$ is $(\alpha,\beta)$-satisfying.
\end{lemma}

\begin{proof}
Again, we can assume by scaling that $N_S=\sigma(S)$ for $S\in\F$ are all integers. Let $\F'$ be the multi-set system where each $S \in \F$ is repeated $N_S$ times, so that $\sum N_S>\kappa^w$. We may also assume that all sets in $\F'$ have size exactly $w$, by adding different dummy elements to each set of size less than $w$. We take care to scale by a large enough factor so that a negligible amount of weight falls on each dummy element, and so the spreadness hypothesis is preserved. Let $N=\sum N_S \ge s_0$ and let $\F'=\{S_1,\ldots,S_N\}$. If $\sigma'$ is the uniform weighting on sets of $\F'$, then $(\F',\sigma')$ also satisfies the assumption of the lemma.  Note also that for any set $W \subset X$, if $W$ contains a set of $\F'$ then it also contains a set of $\F$, so it suffices to verify the $(\alpha,\beta)$-satisfyingness of $\F'$.

The proof is by Janson's inequality (see for example \cite[Theorem 8.1.2]{alon2004probabilistic}). Let $W \sim \U(X,\alpha)$ and let $\mathcal Z_i$ be the indicator variable for $S_i \subset W$. Denote $i \sim j$ if $S_i,S_j$ intersect. Define
$$
\mu = \sum_i \E[\mathcal Z_i], \qquad \Delta = \sum_{i \sim j} \E[\mathcal Z_i \mathcal Z_j].
$$
We have $\mu = N \alpha^w$. To compute $\Delta$, let $p_{\ell}$ denote the fraction of the $N^2$ pairs $(i,j)$ with $1 \le i,j \le N$ such that $|S_i \cap S_j|=\ell$. Then
$$
\Delta = \sum_{\ell=1}^w p_{\ell} N^2 \alpha^{2w-\ell}. 
$$
To bound $p_{\ell}$, note that for each $S_i \in \F$, and any $R \subset S_i$ of size $|R|=\ell$, the number of $S_j \in \F$ such that $R \subset S_j$ is $|\F_R| \le N / \kappa^{|R|}$ by the spreadness assumption. Let $q=\max\left\{4\log(1/\beta),2\right\}$. Thus we can bound
$$
\Delta \le \sum_{\ell=1}^w {\binom w \ell} \kappa^{-\ell} N^2 \alpha^{2w-\ell} 
\le \sum_{\ell=1}^w \left(\frac w {\alpha \kappa}\right)^{\ell} \mu^2 = \mu^2\sum_{\ell=1}^w \left(\frac 1q\right)^{\ell} < \frac{2 \mu^2}{q}.
$$
Note that in addition $\Delta \ge \mu$, as we include the pairs $(i,i)$ in the summation for $\Delta$. Thus by Janson's inequality,
$$
\Pr[\forall i, \mathcal Z_i=0] \le \exp\left\{-\frac{\mu^2}{2 \Delta}\right\} < \exp\left\{-\frac q4\right\} \le \beta.
$$
\end{proof}

\subsection{Putting everything together}

We prove \Cref{thm:kappa_bound} in this subsection, where our goal is to bound $\kappa(w,\alpha,\beta)$. We will apply \Cref{lemma:reduction} iteratively,
until we reach a case where we can apply \Cref{lemma:final}. 

Let $w \ge 2$ be fixed throughout, and let $\kappa > 1$, to be optimized later. 
We first introduce some notation. For $0<\Delta<1$ and  $\ell \ge 1$, let $\s(\Delta,\ell)=(1-\Delta; \kappa^{-1},\ldots,\kappa^{-\ell})$ be a weight profile. Let $A(\Delta,\ell), B(\Delta,\ell)$ be such that any $\s(\Delta,\ell)$-spread set system is $(A(\Delta,\ell), B(\Delta,\ell))$-satisfying.

\Cref{lemma:reduction} applied to $w' \ge w''$ and $p,\delta$ shows that we may take 
\begin{align*}
&A(\Delta, w') \le A(\Delta+\delta, w'') + p,\\
&B(\Delta, w') \le B(\Delta+\delta, w'') + \frac{(4/p)^{w'}}{\delta (1-\Delta) \kappa^{w''}}.
\end{align*}
We apply this iteratively for some widths $w_0,\ldots,w_r$. Set $w_0=w$ and $w_{i+1} = \lceil (1-\eps) w_i \rceil$ for some small $\eps$ as long as $w_i>w^*$ for some $w^*$. In particular, we need $w^* \ge 1/\eps$ to ensure $w_{i+1}<w_i,$ and we will optimize $\eps$ and $w^*$ later. The number of steps is thus $r \le (K\log w) / \eps$ for some constant $K > 0$. Let $p_1,\ldots,p_r$ and $\delta_1,\ldots,\delta_r$ be the values we use for $p,\delta$ at each iteration. To simplify the notation,
let $\Delta_i=\delta_1+\cdots+\delta_i$ and $\Delta_0=0$. Furthermore, define
$$
\gamma_i = \frac{(4/p_i)^{w_{i-1}}}{\kappa^{w_i}}.
$$
Then for $i=1,\ldots,r$, we may take
\begin{align*}
&A(\Delta_{i-1}, w_{i-1}) \le A(\Delta_i, w_i) + p_i,\\
&B(\Delta_{i-1}, w_{i-1}) \le B(\Delta_i, w_i) + \frac{\gamma_i}{\delta_i (1-\Delta_{i-1}) }.
\end{align*}

Set $p_i = p= \frac{\alpha}{2r}$ and $\delta_i = \sqrt{\gamma_i}$, where $r\leq(K\log w)/\eps$ is the number of steps. We will select the parameters so that $\Delta_i \le 1/2$ for all $i$. Thus we may take $A(1/2,w^{*})$ and $B(1/2,w^{*})$ such that
\begin{align*}
&A(0, w) \le A(\Delta_r,w_r) + \alpha/2 \le A(1/2, w^*) + \alpha/2,\\
&B(0, w) \le B(\Delta_r,w_r) + 2 \Delta_r \le B(1/2, w^*) + 2 \Delta_r.
\end{align*}
Plugging in the values for $\delta_i$, we compute the sum
\begin{align*}
\Delta_r 
&= \sum_{i=1}^r \delta_i 
\le \sum_{i=1}^r \sqrt{\frac{(4/p)^{w_{i-1}}}{\kappa^{(1-\eps)w_{i-1}}}} 
= \sum_{i=1}^r \left(\frac{4/p}{\kappa^{1-\eps}}\right)^{w_{i-1}/2}
\le \sum_{k \ge w^*} \left( \frac{8K \log w}{\eps \alpha \kappa^{1-\eps}} \right)^{k/2}\\
&=\left( \frac{8K \log w}{\eps \alpha \kappa^{1-\eps}} \right)^{w^*/2} \sum_{k \ge 0} \left( \frac{8K \log w}{\eps \alpha \kappa^{1-\eps}} \right)^{k/2}
\le 2^{1-w^*},
\end{align*}
assuming $\kappa^{1-\eps} \ge (32K\log w)/\eps \alpha$. 

Next, we will apply \Cref{lemma:final} to bound $A(1/2,w^*) \le \alpha/2$ and $B(1/2,w^*) \le \beta/2$. We use the simple observation that $(1/2; \kappa^{-1},\ldots,\kappa^{-w^*})$-spread set systems are also $(1; (\kappa/2)^{-1},\ldots,(\kappa/2)^{-w^*})$-spread, in which case we can apply \Cref{lemma:final} and obtain that we need
$$
\kappa/2 \ge (2+4\log(2/\beta))\cdot 2w^*/\alpha.
$$
Let us now put the bounds together. We still have the freedom to choose $\eps>0$ and $w^* \ge 1/\eps$. To obtain $A(0,w) \le \alpha, B(0,w) \le \beta$, we also need $\Delta_r \le \beta/4 < 1/2$. Thus all the constraints are:
\begin{enumerate}
\item $w^* \ge 1/\eps$;
\item $\kappa^{1-\eps} \ge (32K\log w) / \eps \alpha$ where $K>0$ is a constant;
\item $\kappa \ge  4\cdot(2+4\log(2/\beta))\cdot w^*/\alpha$;
\item $2^{1-w^*} \le \beta/4$.
\end{enumerate}
Set $\eps = 1/\log \log w$ and $w^* = c \cdot \max\left\{\log(1/\beta), \log \log w \right\}$ for some large enough constant $c \ge 1$. This ensures that the first and last conditions hold.

Thus we obtain that the result holds whenever
$$\kappa=\Omega \left( \max\left\{ 
\left(\frac1\alpha\right)^{1+2/\log\log w}\log w \log \log w,
\frac1\alpha\left(1+\log \frac1\beta\right)^2,
\frac1\alpha\left(1+\log \frac1\beta\right)\log\log w
\right\} \right).
$$
In particular, it suffices to set $\kappa = O\left(\frac{1}{\alpha^2} \cdot \left(\log w \log \log w+ \left(\log\frac 1\beta\right)^2\right)\right)$.

\section{A Lower Bound for Robust Sunflowers}
\label{sec:lower_bound}
In this section, we construct an example of a $w$-set system of size $(\log w)^{w(1-o(1))}$ which does not contain a robust sunflower, showing that \Cref{thm:robust_sunflower} is tight up to the $o(1)$ in exponent.

For concreteness we fix $\alpha=\beta=1/2$, but the construction can be easily modified for any other constant values of $\alpha,\beta$.
We assume that $w$ is sufficiently large.

\begin{lemma}
\label{lemma:lower_bound}
There exists a $w$-set system of size $((\log w) / 8)^{w - \sqrt{w}} = (\log w)^{w(1-o(1))}$ which does not contain a $(1/2,1/2)$-robust sunflower.
\end{lemma}

Let $c \ge 1$ be determined later. Let $X_1,\ldots,X_w$ be pairwise disjoint sets of size $m=\log(w/c)$, and let $X$ be their union. Let $\widehat\F = X_1 \times \cdots \times X_w$ be the $w$-set system consisting of all sets which contain exactly one element from each of the $X_i$. We first argue that $\widehat\F$ is not satisfying.

\begin{claim}
\label{claim:F0}
For $c \ge 1$, $\widehat\F$ is not $(1/2,1/2)$-satisfying.
\end{claim}

\begin{proof}
Let $R \sim \U(X,1/2)$. The probability that some $X_i$ is disjoint from $R$ is $1 - (1-2^{-m})^w = 1-(1-c/w)^w>1/2$ if $c \ge 1$.  Thus if $c \ge 1$, with probability more than $1/2,$ $R$ is disjoint from some $X_i$ and thus does not contain any set $S \in \widehat\F$. \end{proof}

The family $\widehat\F$ does contain a $(1/2,1/2)$-robust sunflower. For example, if $T$ contains exactly one element from each of $X_1,\ldots,X_{w-1}$, then $\widehat\F_T$ is isomorphic to $X_w$, and in particular is certainly $(1/2,1/2)$-satisfying.  However, it turns out that the robust sunflowers in $\widehat\F$ all contain large kernels.  We will exploit this to construct a large subsystem of $\widehat\F$ which does not contain a robust sunflower. Let $\eps>0$ be determined later, and choose $\F \subset \widehat\F$ to be a subsystem that satisfies:
$$
|S \cap S'| \le (1-\eps)w, \qquad \forall S,S' \in \F, S \ne S'.
$$
For example, we can obtain $\F$ by a greedy procedure, each time choosing an element $S$ in $\widehat \F$ and deleting all $S'$ such that $|S \cap S'| > (1-\eps) w$. The number of such $S'$ is at most ${\binom w {\eps w}} m^{\eps w} \le 2^w m^{\eps w}$. Hence we can obtain $\F$ of size $|\F| \ge 2^{-w} m^{(1-\eps)w}$. 

\begin{claim}
\label{claim:F1}
For $c \ge 1/\eps$, $\F$ does not contain a $(1/2,1/2)$-robust sunflower.
\end{claim}

\begin{proof}
Consider any set $K \subset X$. We need to prove that $\F$ does not contain a $(1/2,1/2)$-robust sunflower with kernel $K$. If it does, $\F_K$ must contain at least two sets $S$ and $S'$, which implies that $|K \cap X_i| \le |S \cap X_i| \le 1$ for all $i$.  Furthermore $K \subset S \cap S'$ and so $|K| \le |S \cap S'| \le (1-\eps)w$.  However, for such $K$ we claim even $\widehat\F_K$ is not $(1/2,1/2)$-satisfying.

To prove this, let $I=\{i: |K \cap X_i| = 0\}$, so that $|I| \ge \eps w$. Let $R \sim \U(X,1/2)$. The probability that there exists $i \in I$ such that $R$ is disjoint from $X_i$ is $1 - (1-2^{-m})^{|I|} \ge 1 - (1-c/w)^{\eps w}$ which is more than $1/2$ for $c \ge 1/\eps$.
\end{proof}

To conclude the proof of \Cref{lemma:lower_bound} we 
optimize the parameters. Set $c=1/\eps$. We have $|\F| \ge 2^{-w} (\log (\eps w))^{(1-\eps) w}$. Setting $\eps=1/\sqrt{w}$ gives $|\F| \ge ((\log w) / 8)^{w-\sqrt{w}} = (\log w)^{(1-o(1))w}$.

\section{Subsequent Works and Applications}
\label{sec:applications}

After the current work was made available on the ArXiv, Rao \cite{rao2019coding} simplified the proof of \Cref{thm:kappa_bound} using information theoretic techniques, and obtained a slightly better dependence on $\alpha$ and $\beta$. Furthermore, following the initial release of this work, the technique developed in this paper has been used by Frankston, Kahn, Narayanan, and Park~\cite{fknp} to resolve a conjecture of Talagrand in random graph theory. Rao then used their refinements to further improve the bound in \Cref{thm:kappa_bound} to $(C/\alpha) \log(w/\beta)$ for some constant $C$, hence improving the bound in \Cref{thm:robust_sunflower} to $\left((C/\alpha) \log(w/\beta)\right)^w$ and the bound in \Cref{thm:sunflower} to $\left(Cr\log(wr)\right)^w$.
Bell, Chueluecha, and Warnke \cite{bell2021note} observed that a small modification of the argument improves the bound in \Cref{thm:sunflower} further to $\left(Cr\log(w)\right)^w$.

Following Rao's proof using information theory, two more proofs using information theory were given by Tao \cite{tao-blog} and Hu \cite{hu-blog}.
Fox, Pach, and Suk \cite{fox2021sunflowers} studied the sunflower conjecture in the special setting where the set system has bounded VC dimension, and were able to dramatically improve the bounds in this case, getting very close to the conjectured bound.

On the computer science side, \Cref{thm:robust_sunflower} has been used by Cavalar, Kumar, and Rossman~\cite{cavalarmonotone} to improve previous monotone circuit lower bounds.  There are further applications of this work in theoretical computer science; see the extended abstract version of this paper \cite{ALWZstoc} for details.

We present several additional applications of \Cref{thm:kappa_bound} below. To present sharper (and cleaner) bounds, we instead use the version proved by Rao \cite{rao2019coding}, building on the work of Frankston, Kahn, Narayanan, and Park~\cite{fknp}. Using the notations of \Cref{thm:kappa_bound}, he proved that  $\kappa(w,\alpha,\beta) \le (C/\alpha)\log(w/\beta)$ for some constant $C$.

\paragraph{The Erd\H{o}s-Szemer\'edi sunflower conjecture.}
In \cite{es}, the Erd\H{o}s-Rado sunflower conjecture is shown to imply the Erd\H{o}s-Szemer\'edi sunflower conjecture.  Plugging in our new bounds into their argument gives a corresponding improvement. 

\begin{theorem}\label{thm:sunflower_ES_new}
For any $r \ge 3$ there exists $c=c(r)$ such that the following holds. Let $\F$ be a set system on $X$, with $|X|=n$ and $|\F| \ge 2^{n\left(1-c / \log n \right)}$. Then $\F$ contains an $r$-sunflower.
\end{theorem}
The Erd\H{o}s-Szemer\'edi sunflower conjecture says the bound on the size of $\F$ can be improved to $2^{n(1-\eps)}$ for some $\eps=\eps(r)$. This was proved for $r=3$ by Naslund \cite{naslund2017upper} using algebraic techniques, which so far do not seem to generalize for larger values of $r$.

\paragraph{Intersecting set systems.}
A set system is \emph{intersecting} if any pair of its sets intersects. We note the following corollary of \Cref{thm:kappa_bound}:

\begin{theorem}
\label{thm:interset}
If $\F$ is an intersecting $w$-uniform set system, and for all $T$, $|\F_T| \le \kappa^{-|T|}|\F|$, then $\kappa=O(\log w)$.
\end{theorem}

An example from \cite{lovett2019dnf} shows that  for $\kappa=\Omega(\log w/ \log \log w)$, there are intersecting $\kappa$-spread $w$-uniform set systems, so the bound in  \Cref{thm:interset} is close to tight.

\begin{proof} If $\F$ is intersecting then it is not $(1/2,1/2)$-satisfying (apply \Cref{lemma:satisfying_contains_disjoint} for $r=2$). Thus by the improvement of \Cref{thm:kappa_bound} from \cite{rao2019coding}, it cannot be $(C \log w)$-spread for a large enough constant $C$.
\end{proof}

Note that \Cref{thm:interset} is fairly robust in the sense that the bounds on the links of $T$ for large $T$ are required.  For instance, for a small constant $\eps$ a family $\F$ of $2^{n^{\eps}}$ random subsets of $[n]$ of size $w=n^{1/2+\eps}$ is intersecting with high probability.  Also with high probability it is $(1; \kappa, \ldots, \kappa^{-\ell})$-spread for $\ell$ up to $n^{\eps/2}=w^{\Omega(1)}$, where $\kappa=w^{\Omega(1)}$.  This example shows that in \Cref{thm:kappa_bound} the spreadness condition on the large sets is necessary.  Thus it seems difficult to find a purely algebraic or analytic proof of \Cref{thm:kappa_bound}, as one would have to use the spreadness condition on large sets in a useful but non-combinatorial way.

We expect \Cref{thm:interset} to have a fair number of applications. Below we describe two applications: one to packing of Kneser graphs, and the other to the Alon-Jaeger-Tarsi conjecture.

\paragraph{Packing of Kneser graphs.}
Let $n > k \ge 1$. The Kneser graph $\text{KG}(n,k)$ has as vertices all the subsets of $[n]$ of size $k$, and its edges correspond to pairs of disjoint sets.

\begin{theorem}
\label{thm:kpacking}
Let $n > k \ge 1$. If two copies of the Kneser graph $\text{KG}(n,k)$ can be packed into the complete graph on $\dbinom{n}{k}=N$ vertices, then $k=\Omega\left(\frac{n}{(\log n)^2}\right)$.
\end{theorem}

\begin{proof} 
Let $G=\text{KG}(n,k)$.
We identify the vertex set $V=V(G)$ with the subsets of $[n]$ of size $k$. By the assumption, there is a bijection $\pi:V \to V$ such that if $(u,v) \in E(G)$ then $(\pi(u),\pi(v)) \notin E(G)$. 

Define a $2k$-uniform hypergraph $H$ as follows. Let $V(H)=V_0 \cup V_1$ where $V_b = \{(b,i): i \in [n]\}$. The edges of $H$ are $E(H)=\{e(v): v \in V\}$ where $e(v)=\{(0,i): i \in v\} \cup \{(1,i): i \in \pi(v)\}$ for $v \in V$.

We claim that $H$ is intersecting. Fix any $u,v \in V$ and consider the corresponding edges $e(u),e(v)$. Observe that by assumption either $(u,v) \notin E(G)$ or $(\pi(u),\pi(v)) \notin E(G)$. Thus either $u,v$ intersect or $\pi(u),\pi(v)$ intersect. In either case, $e(u),e(v)$ intersect.

Hence, by \Cref{thm:interset}, for a large constant $C$ there will be a subset $U \subset V(H)$ so that at least a $(C \log 2k)^{-|U|}$ fraction of the edges of $H$ contain $U$.  Assume without loss of generality that $|U \cap V_0| \ge |U|/2$.  The fraction of the edges of $H$ that contain any subset of $V_0$ of size at least $|U|/2$ is at most 
$$
\frac{\dbinom{n-|U|/2}{k-|U|/2}}{ \dbinom{n}{k}} \le (k/n)^{|U|/2}.
$$  
We thus obtain that $(k/n)^{|U|/2} \ge (C \log 2k)^{-|U|}$ and so $k=\Omega\left(\frac{n}{(\log n)^2}\right)$ as claimed. \end{proof}

We believe that assuming the premise of \Cref{thm:kpacking}, it should be true that $k=\Omega(n)$.  This problem is of independent interest.

\begin{conjecture}
\label{conj:kpacking}
Let $n > k \ge 1$. If two copies of the Kneser graph $\text{KG}(n,k)$ can be packed into the complete graph on $\dbinom{n}{k}=N$ vertices, then $k=\Omega\left(n\right)$.
\end{conjecture}

\paragraph{The Alon-Jaeger-Tarsi conjecture.}
Let $A,B$ be two $n \times n$ non-singular matrices over a finite field $\mathbb{F}_p$. The Alon-Jaeger-Tarsi conjecture \cite{alon1989nowhere} is that for $p \ge 4$, there must exist $x \in \mathbb{F}_p^n$ such that both $Ax$ and $Bx$ have only nonzero coordinates. The analogous statement for prime power fields was proved in \cite{alon1989nowhere}. Yu \cite{yu1999permanent} proved that this conjecture is true if $p \ge \log n$. Using \Cref{thm:interset}, we prove an extension for $p \ge (\log n)^{O(1)}$.

\begin{theorem}
Let $r \ge 2$. Let $A_1,\ldots,A_r$ be non-singular $n \times n$ matrices over a field $\mathbb{F}_p$. Then, for some absolute constant $C$, as long as $p>(C \log(rn))^r$, there exists $x \in \mathbb{F}_p^n$ such that $A_1 x,\ldots,A_r x$ have only nonzero coordinates. 
\end{theorem}

\begin{proof}
Let $w=rn$. Identify $x \in \mathbb{F}_p^n$ with the $w$-set 
$$
S_x = \left\{(i,j,(A_i x)_j): i \in [r], j \in [n]\right\}.
$$
In particular, $S_x$ is a set of $w$ triples.

Let $\F=\{S_x: x \in \mathbb{F}_p^n\}$. The main observation is that if no solution $x$ exists, then $\F$ needs to be an intersecting $w$-uniform set system. Indeed, if $S_{x'},S_{x''}$ are pairwise disjoint then one can verify that $x=x'-x''$ satisfies that $A_1 x,\ldots,A_r x$ have all nonzero coordinates.

Next, we analyze the size of links of $\F$. A link $\F_T$ corresponds to solutions to a system of $|T|$ linear equations of the form $(A_i x)_j = a_{i,j}$. The number of linearly independent equations is at least $|T|/r$, because at least $|T|/r$ of them must come from one of the matrices $A_i$, and hence we obtain
$$
|\F_T| \le p^{-|T|/r} |\F|.
$$
Thus we obtain that 
$p^{1/r} \le (C \log w)$,
which is a contradiction. \end{proof}

We note that very recently Nagy and Pach \cite{nagy2021alon} proved the Alon-Jaeger-Tarsi conjecture using algebraic techniques.

\section{Rainbow Sunflowers}
\label{sec:rainbow}

In this section, we discuss a possible line of attack toward the sunflower conjecture (\Cref{conj:sunflower}).  We propose the following more general conjecture.

\begin{conjecture}[Rainbow sunflower conjecture]
\label{conj:rainbowsunflower}
There exists a constant $C$ such that the following holds. Let $\F$ be a $w$-set system of size $|\F| \ge C^w$ over a ground set $X$. Assume that the elements of $X$ are colored with either red, green, or blue uniformly and independently at random. 

Then with high probability, there exists distinct sets $S_i, S_j, S_k \in \F$ such that if $Y=S_i \cap S_j \cap S_k$, then all elements of $S_i \setminus Y$ are red, all elements of $S_j \setminus Y$ are green, and all elements of $S_k \setminus Y$ are blue.
\end{conjecture}

Note that \Cref{conj:rainbowsunflower} implies \Cref{conj:sunflower} for $r=3$. 
Furthermore, this paper gives a proof of \Cref{conj:rainbowsunflower} where $C^w$ is replaced by $(\log w)^{w(1+o(1))}$, and Rao \cite{rao2019coding} improved the bound to $(C \log w)^w$ for some constant $C$.

Unlike in \Cref{thm:robust_sunflower}, we do not know if $(C\log w)^{w}$ is tight, and in fact we suspect this is not the case.  \Cref{conj:rainbowsunflower} is in some sense similar to \Cref{thm:robust_sunflower}; the latter differs in that one must commit to a subfamily before seeing the coloring, whereas in the former the sets $S_i, S_j, S_k$ may depend on the coloring.  We propose another variant of \Cref{conj:rainbowsunflower}.

\begin{conjecture}
\label{conj:rainbow2petals}
There exists a constant $C$ such that the following holds.
Let $\F$ be a $w$-set system of size $|\F| \ge C^w$ over a ground set $X$. Assume that the elements of $X$ are colored with either red or blue uniformly and independently. In addition let $S \in \F$ be sampled uniformly.

Then with high probability,  there exists distinct sets $S_i, S_j \in \F$ such that all elements of $S_i \setminus S$ are red, all elements of $S_j \setminus S$ are blue, and $S_i \cap S=S_j \cap S\subsetneq S$.
\end{conjecture}

\Cref{conj:rainbow2petals} also implies \Cref{conj:sunflower} for $r=3$, but we do not know of any nontrivial bounds for it.  It would be interesting to have even a $w^{o(w)}$ bound.
We can however prove the following lemma, which provides some evidence that \Cref{conj:rainbowsunflower} might be true.

\begin{lemma}
\label{lemma:rainbow2petal}
There exists a constant $C$ such that the following holds.
Let $\F$ be a $w$-set system of size $|\F| \ge C^w$ over a ground set $X$. Assume that the elements of $X$ are colored with either red or blue uniformly and independently.

Then with high probability, there exists distinct sets $S_i, S_j \in \F$ such that if $Y=S_i\cap S_j$, then all elements of $S_i\setminus Y$ are red and all elements of $S_j\setminus Y$ are blue (equivalently, all elements of $S_i \setminus S_j$ are red and all elements of $S_j \setminus S_i$ are blue).
\end{lemma}
 
In order to establish this, we first prove a weaker result along the same lines.

\begin{lemma}
\label{lemma:itoj}
For all $\delta>0$ there exists $C=C(\delta)$ such that the following holds.
Let $\F$ be a $w$-set system of size $|\F| = m \ge C^w$ over a ground set $X$. Assume that a random $\delta$-fraction of the elements of $X$ are colored white.
Then:
\begin{enumerate}
    \item The expected number of sets $S_j \in \F$, for which there is no $i \ne j$ with $S_i \setminus S_j$ all white, is at most $C^w$.
    
    \item With high probability, for almost all $S_j \in \F$ there are at least $m / C^w$ many sets $S_i \in \F$ such that $S_i \setminus S_j$ is all white.
    \end{enumerate}
\end{lemma}

\begin{proof}
We will show that the lemma holds for $C(\delta)=4/\delta$.
Let $W$ denote the white subset of $X$.  Call a pair $(S_j,W)$ \emph{bad} if there is no $i \neq j$ such that $S_i \setminus S_j \subset W$.  If $(S_j, W)$ is a bad pair, then $S_j$ is the only set of $\F$ contained in $S_j \cup W$.  Thus, $S_j \cup W$ uniquely determines $S_j$, and another $w$ bits encoding $S_j \cap W$ determine the pair $(S_j,W)$.  Hence, the number of bad pairs $(S_j,W)$ is only a factor of $(2/\delta)^w$ more than the number of possible $W$.  Thus, for a random $W$, in expectation only $(2/\delta)^w<C^w$ of the sets $S_j \in \F$ are bad.

For the second part of the lemma, redefine bad so that $(S_j,W)$ is bad if there are less than $K=m/C^w$ values of $i \neq j$ such that $S_i \setminus S_j \subset W$.  Then each bad pair $(S_j,W)$ can be identified by $S_j \cup W$ and a positive integer up to $K$, so in expectation only $K(2/\delta)^w=m\cdot 2^{-w}$ of the sets $S_j \in \F$ are bad.
Then by Markov's inequality, with probability at least $1-2^{-w/2}$, for a $1-2^{-w/2}$ fraction of $S_j\in\F$ there are at least $m/C^w$ many sets $S_i\in\F$ such that $S_i\setminus S_j$ is all white.
\end{proof}

By adding dummy elements to $X$, \Cref{lemma:itoj} also holds with each element of $X$ independently colored white with probability $\delta$. Now we prove \Cref{lemma:rainbow2petal}.

\begin{proof}[Proof of \Cref{lemma:rainbow2petal}]
We take a random red-blue coloring of $X$, where we can treat either the red or the blue elements as the white elements in \Cref{lemma:itoj}, and then apply \Cref{lemma:itoj} with $\delta=1/2$.

Choose a random $S\in\F$ and fix it for the remainder of the proof. By the second part of \Cref{lemma:itoj}, in a random coloring with high probability there are at least $m/C^w$ many sets $S_i \in \F$ so that $S_i \setminus S$ is all red.

Call a set $S_i$ \emph{good} if there exists $S_j$ so that $S_j \setminus S_i$ is blue and $S_j \cap S=S_i \cap S$.  Call it \emph{bad} otherwise (this ``bad'' is different from the one in \Cref{lemma:itoj}).  For each $Y \subset S$, consider the family $\F_{S,Y}$ of sets in $\F$ that intersect $S$ exactly at $Y$.  By the first part of \Cref{lemma:itoj}, in expectation only $C^w$ of the sets $S_i \in \F_{S,Y}$ will be bad.  Taking the union over all possible $Y$, it follows in expectation that only $(2C)^w$ of the sets $S_i \in \F$ will be bad.

So with high probability over the random coloring two events hold: (i) at most $(4C)^w$ of the sets $S_i \in \F$ are bad; and (ii) at least $m/C^w$ of the sets $S_i \in \F$ will have $S_i \setminus S$ all red. So, as long as $m/C^w > (4C)^w$, there is some good $S_i$ so that $S_i \setminus S$ is all red. Then for each such $S_i$ there will be some $S_j$ so that $S_i \setminus S_j \subset S_i \setminus S$ is all red and $S_j \setminus S_i$ is all blue.  This concludes the proof. \end{proof}

The following is another corollary of \Cref{lemma:rainbow2petal}.  It follows from \Cref{lemma:rainbow2petal} in exactly the same way that the Erd\H{o}s-Szemer{\'e}di sunflower lemma follows from the Erd\H{o}s-Rado sunflower lemma (see \cite{es}), so we omit the proof.

\begin{lemma}
\label{lemma:es}
There exists $\delta>0$ such that the following holds. Let $\F$ be a family of subsets of $[n]$ of size $|\F|>(2-\delta)^n$. Color each element of $[n]$ either red or blue uniformly and independently at random. 

Then with high probability, there exists distinct $S_i,S_j\in\F$ such that $S_i \setminus S_j$ is all red and $S_j \setminus S_i$ is all blue.
\end{lemma}

In particular, this implies that if $R$ is a uniform random subset of $[n]$, then with high probability there exist some $S_i, S_j \in \F$ are such that $S_i \oplus R \subset S_j \oplus R$, where $\oplus$ denotes the exclusive or function.

Finally, we note that a statement similar to \Cref{conj:rainbowsunflower} is equivalent to \Cref{conj:sunflower}.

\begin{conjecture}[Monochromatic sunflower conjecture]
\label{conj:monopetal}
For all $\delta>0$ there exists a constant $K=K(r,\delta)$ such that the following holds. Let $\F$ be a $w$-set system of size $|\F| \ge K^w$ on a ground set $X$. Color a random $\delta$-fraction of the elements of $X$ with red. 

Then with high probability, $\F$ contains an $r$-sunflower whose petals are all red.
\end{conjecture}

\begin{theorem}
\label{thm:monopetal_eq_sunflower}
\Cref{conj:monopetal} is equivalent to \Cref{conj:sunflower}.
\end{theorem}

\begin{proof}
Clearly, \Cref{conj:monopetal} implies \Cref{conj:sunflower}, so now we prove the other direction.
Assuming \Cref{conj:sunflower}, let $C= C(r)$ be such that any $w$-set system of size at least $C^w$ contains an $r$-sunflower.  By the second part of \Cref{lemma:itoj}, if $K=K(r,\delta)$ is large enough, and $S \in \F$ is uniformly sampled, then with high probability there are at least $(2C)^w$ sets $S_i$ with $S_i \setminus S$ all red.  Thus there exists $Y \subset S$ such that at least $C^w$ many sets $S_i$ have $S_i \setminus S$ all red and $S_i \cap S=Y$.
In particular, there are $r$ sets $S_1, \ldots, S_r$ forming an $r$-sunflower so that for $1 \le i \le r$ we have $S_i \setminus S$ red and $S_i \cap S=Y$.  Because $Y$ is contained in the kernel of the sunflower, all of the petals are red.
\end{proof}

The above argument also shows that the optimal bounds in \Cref{conj:monopetal} and \Cref{conj:sunflower} are off by an exponential factor in $w$. For $r=3$, we expect a similar relationship between \Cref{conj:rainbowsunflower} and \Cref{conj:sunflower}.

\section*{Acknowledgements}
We thank Noga Alon, Matthew Wales, Andrey Kupavskii, Mohammadmahdi Jahanara, two anonymous STOC reviewers, and two anonymous Annals reviewers for helpful suggestions on an earlier version of this paper. We also thank Andrew Thomason, Gil Kalai, and Peter Frankl for pointing us to relevant references.

\bibliographystyle{abbrv}
\bibliography{sunflower}

\end{document}